\documentclass[12pt]{article}
\usepackage{amssymb}
\usepackage{verbatim}
\usepackage{array}
\usepackage{latexsym}
\usepackage{enumerate}
\usepackage{amsmath}
\usepackage{amsfonts}
\usepackage{amsthm}
\usepackage{graphicx,epsfig}
\usepackage[english]{babel}

\title{{\bf{Some computational results on small 3-nets embedded in a projective plane over a field}}}
\date{}

\author{G.~P.~Nagy and N.~Pace.}

\newtheorem{theorem}{Theorem}[section]
\newtheorem{proposition}[theorem]{Proposition}
\newtheorem{lemma}[theorem]{Lemma}

{\theoremstyle{definition}
\newtheorem*{definition*}{Definition}

\newtheorem*{proposition*}{Proposition}
\newtheorem*{corollary*}{Corollary}
\newtheorem*{lemma*}{Lemma}

\newcommand{\Alt}{\mathrm{Alt}}

\def\C{\mathbf C}

\def\cA{\mathcal A}
\def\cB{\mathcal B}
\def\cC{\mathcal C}

\begin{document}

\maketitle
\begin{abstract}
In this paper, we investigate dual $3$-nets realizing the groups $C_3 \times C_3$, $C_2 \times C_4$,
$\Alt_4$ and that can be embedded in a projective plane $PG(2,\mathbb K)$, where $\mathbb K$ is an
algebraically closed field. We give a symbolically verifiable computational proof that every dual
$3$-net realizing the groups $C_3 \times C_3$ and $C_2 \times C_4$ is algebraic, namely, that its
points lie on a plane cubic. Moreover, we present two computer programs whose calculations show
that the group $\Alt_4$ cannot be realized if the characteristic of $\mathbb K$ is zero. 
\end{abstract}

\section{Introduction}
\label{problem}

In a projective plane a {\em{$3$-net}} consists of three pairwise disjoint classes of
lines such that every point incident with two lines from distinct classes is incident with
exactly one line from each of the three classes. If one of the classes has finite size,
say $n$, then the other two classes also have size $n$, called the {\em{order}} of the
$3$-net. In this paper we are considering $3$-nets in a projective plane $PG(2,\mathbb K)$
over an algebraically closed field $\mathbb K$ which are coordinatized by a group. Such a
$3$-net, with line classes $\cA,\cB,\cC$ and coordinatizing group $G=(G,\cdot)$, is
equivalently defined by a triple of bijective maps from $G$ to $(\cA,\cB,\cC)$, say
$$\alpha:\,G\to \cA,\,\beta:\,G\to \cB,\,\gamma:\,G\to \cC$$ such that $a\cdot b=c$ if and
only if $\alpha(a),\beta(b),\gamma(c)$ are three concurrent lines in $PG(2,\mathbb{K})$,
for any $a,b,c \in G$. If this is the case, the $3$-net in $PG(2,\mathbb{K})$ is said to
{\em{realize}} the group $G$. Recently, finite $3$-nets realizing a group in the complex plane
have been investigated in connection with complex line arrangements and resonance theory,
see \cite{bkm,knp2011,ys2004} and the references therein.

Since key examples arise naturally in the dual plane of $PG(2,\mathbb{K})$, it is
convenient to work with the dual concept of a $3$-net. Formally, a {\em{dual $3$-net}} of
order $n$ in $PG(2,\mathbb{K})$ consists of a triple $(\Lambda_1,\Lambda_2,\Lambda_3)$
with $\Lambda_1,\Lambda_2,\Lambda_3$ pairwise disjoint point-sets of size $n$, called
{\em{components}}, such that every line meeting two distinct components meets each
component in precisely one point. A dual $3$-net $(\Lambda_1,\Lambda_2,\Lambda_3)$
realizing  a group is {\em{algebraic}} if its points lie on a plane cubic. Moreover, we
say that the dual $3$-net $(\Lambda_1,\Lambda_2,\Lambda_3)$ is \emph{of conic-line
type (triangular)} if the components are contained in the union of a line and a
nonsingular conic (in the union of three lines). 

In our computer-aided investigation, combinatorial methods are used to study finite
$3$-nets realizing the groups $\C_2 \times \C_4, \C_3 \times \C_3,$ and $\Alt_4$.
These results are fundamental for the complete classification of 3-nets embedded in a
projective plane over a field, see \cite{knp2011}. Indeed, large groups could be dealt
with theoretical results, but small groups having elements of order less than $5$ needed a
more explicit computation. This is main motivation of this paper.

We can summarize our results in the following theorem.
\begin{theorem}
\label{mainteo} Let $(\Lambda_1,\Lambda_2,\Lambda_3)$ be a dual $3$-net of order $n$ which
realizes a group $G$ in the projective plane $PG(2,\mathbb{K})$ defined over an
algebraically closed field $\mathbb{K}$ of characteristic $p$, where $p=0$ or $p\geq 5$.
We also assume that $n<p$ whenever $p>0$. The following statements hold. 
\begin{itemize}
\item[\rm{(I)}] If $G \cong \C_3 \times \C_3$ or $G \cong \C_2 \times \C_4$, then
$(\Lambda_1, \Lambda_2, \Lambda_3)$ is algebraic.
\item[\rm{(II)}] If $p=0$, then the group $\Alt_4$ cannot be realized.
\end{itemize}
\end{theorem}

The proof is divided in three parts, see Sections \ref{sec_c3c3},\ref{sec_c2c4}, or \ref{sec_alt4},
according as $G\cong \C_3 \times \C_3$, $G\cong \C_2 \times \C_4$, or $G\cong \Alt_4$. Our notation
and terminology are standard, see \cite{hkt}. In view of Theorem \ref{mainteo}, $\mathbb{K}$ denotes
an algebraically closed field of characteristic $p$ where either $p=0$ or $n<p$ where $n$ denotes
the order of the dual $3$-net.

\section{$G \cong \C_3 \times \C_3$} \label{sec_c3c3}

We denote by $G=\{0,\ldots,8\}$ the elementary abelian group of order $9$ given by the
multiplication table
\[\begin{array}{cccccccccc}
 \\ &0&1&2&3&4&5&6&7&8\\
&1&2&0&4&5&3&7&8&6\\
&2&0&1&5&3&4&8&6&7\\
&3&4&5&6&7&8&0&1&2\\
&4&5&3&7&8&6&1&2&0\\
&5&3&4&8&6&7&2&0&1\\
&6&7&8&0&1&2&3&4&5\\
&7&8&6&1&2&0&4&5&3\\
&8&6&7&2&0&1&5&3&4
\end{array}.\]
Let $H$ be the subgroup $\{0,1,2\}$ of $G$.

Let $\mathbb{K}$ be an algebraically closed field whose characteristic is either $0$ or
more than $9$. In this paper, all points are points of the projective plane over
$\mathbb{K}$. We denote by $\omega$ a cubic root of unity in $\mathbb{K}$.

It is easy to see that realizations of the cyclic group of order $3$ are precisely the
Pappus configurations. The point set $\mathcal P$ of a triangular dual $3$-net realizing
$\C_3$ consists of $9$ points such that any line intersects $\mathcal P$ in $1$ or $3$
points. In other words, $\mathcal P$ forms an $AG(2,3)$, where $AG(n,q)$ denotes the
affine geometry over the finite field $\mathbb F_q$. It is also well known that any
$AG(2,3)$ embedded in $PG(2,\mathbb{K})$ is a \emph{Hesse configuration,} that is, the
points are the inflection points of a nonsingular cubic curve. 

\begin{lemma} \label{lm:pencil}
Let
\[\Delta'=\{0_1,1_1,2_1, 0_2,1_2,2_2, 0_3,1_3,2_3\}\]
be a realization of $H$. Then there is a unique cyclic collineation $\alpha$ of order
three mapping
\[0_1,1_1,2_1, 0_2,1_2,2_2, 0_3,1_3,2_3 \mbox{ to } 1_1,2_1,0_1, 1_2,2_2,0_2,
2_3,0_3,1_3,\]
respectively. $\alpha$ is never central. The cubic curves containing $\Delta$ form a
pencil. All these cubics are invariant under $\alpha$. \qed
\end{lemma}

\begin{lemma} \label{lm:ag23char}
Let $X$ be a set of nine points in a projective plane such that for all $A,B\in X$, the
line $AB$ contains a third point of $X$. Then, $X$ is either contained in a line, or form
an $AG(2,3)$. \qed
\end{lemma}

In the sequel, we denote by $\Delta=\{0_1,\ldots,8_1, 0_2,\ldots,8_2, 0_3,\ldots,8_3\}$ a
realization of $G$. We denote by $\Delta'$ the subset of $\Delta$ realizing the subgroup
$H=\{0,1,2\}$. We will often use that the points of $\Delta$ can be re-indexed and the
blocks $\{i_1\}$, $\{j_2\}$, $\{k_3\}$ can be interchanged.

\begin{lemma}
There is a line which intersects $\Delta$ in exactly two points.
\end{lemma}
\begin{proof}
Assume that no line intersects $\Delta$ in exactly two points. As $\mathbb{K}^*$ has no
elementary
abelian subgroup of order $9$, $\Delta$ cannot be triangular or of conic-line type.
Theorem 5.1 of \cite{bkm} implies that none of the blocks $\{i_1\}$, $\{j_2\}$, $\{k_3\}$  is
contained in a line. Lemma \ref{lm:ag23char} implies that these blocks must form an
$AG(2,3)$. Moreover, each line intersecting $\Delta$ in more than two points, intersect
$\Delta$ in precisely three points. This means that with respect to the line
intersections, $\Delta$ forms a Steiner triple system. As any three points of $\Delta$
generate a subsystem of order $9$, $\Delta$ is in fact a Hall triple system, cf. \cite[pages
496--499]{crchb}. As
$|\Delta|=27$, we obtain that $\Delta$ is an embedding of $AG(3,3)$ in a projective plane,
which is not possible by \cite{tani2000} if $\mathrm{char}(\mathbb K)\neq 3$. 
\end{proof}

By re-indexing $\Delta$, we can suppose that the line $0_11_1$ intersects $\Delta$ in
$\{0_1,1_1\}$, that is, $0_1,1_1,2_1$ are not collinear. Let $\alpha$ be the
cyclic collineation of order three corresponding to the subnet $\Delta'$ realizing
$H=\{0,1,2\}$. We will choose our projective coordinate system such that the following
hold:
\begin{enumerate}
\item[(1)] $0_1=(1,0,0)$, $1_1=(0,1,0)$ and $2_1=(0,0,1)$.
\item[(2)] $F=(1,1,1)$ is a fixed point of $\alpha$.
\item[(3)] If the lines $0_10_2$, $1_11_2$, $2_12_2$ are concurrent then $F=(1,1,1)$ is
their intersection.
\end{enumerate}

Notice that the lines $i_1 j_1$ contain no fixed point of $\alpha$, hence (2) does not
conflict with (1). Furthermore, if $0_10_2$, $1_11_2$, $2_12_2$ are concurrent then their
intersection is a fixed point of $\alpha$.

The collineation $\alpha$ has the matrix form
\[\left ( \begin{array}{ccc}0&0&1\\1&0&0\\0&1&0\end{array} \right ).\]
As $0_3$ is not on the lines $1_i1_j$, $0_3$ has coordinates of the form $(a,b,1)$ with
$a,b \neq 0$. Then, we can compute the coordinates of the points $1_3=(b,1,a)$,
$2_3=(1,a,b)$ and $0_2=(b, b a, a), 1_2=(a, b, b a), 2_2=(b a, a, b)$.

Let $\beta$ be the cyclic collineation of order $3$ corresponding to the subnet
$\{0_1,1_1,2_1,3_2,4_2,5_2,3_3,4_3,5_3\}$. The matrix of $\beta$ has the form
\[\left ( \begin{array}{ccc}0&0&u\\v&0&0\\0&1&0\end{array} \right )\]
for some nonzero $u,v\in \mathbb{K}$. The point $3_2$ has nonzero coordinates $(x,y,1)$.
Then, we have $4_2=(u,vx,y)$, $5_2=(uy,uv,vx)$ and
\[3_3=(uy,vxy,vx), \; 4_3=(xy,vx,y), \; 5_3=(uvx,vuy,vxy).\]

For all points $i_1$, $i\in \{3,\ldots,8\}$, there are three lines of the form $ j_2 k_3$,
$j,k\in \{0,\ldots,5\}$ such that $i_1 \in j_2 k_3$. The fact that the corresponding line
triples are concurrent, can be expressed by the equations
\begin{eqnarray*}
\hat f_3&=&\det(0_2 \times 3_3, 1_2\times 4_3,2_2\times 5_3)=0,\\
\hat f_4&=&\det(0_2 \times 4_3, 1_2\times 5_3,2_2\times 3_3)=0,\\
\hat f_5&=&\det(0_2 \times 5_3, 1_2\times 3_3,2_2\times 4_3)=0,\\
f_6&=&\det(3_2 \times 0_3, 4_2\times 1_3,5_2\times 2_3)=0,\\
f_7&=&\det(3_2 \times 1_3, 4_2\times 2_3,5_2\times 0_3)=0,\\
f_8&=&\det(3_2 \times 2_3, 4_2\times 0_3,5_2\times 1_3)=0.
\end{eqnarray*}
The values $a,b,u,v,x,y$ determine $\Delta$ uniquely. The $f_i$'s ($i\in
\{3,\ldots,8\}$) are polynomial expressions of these values. In fact, we will look at
$a,b,u,v,x,y$ as indeterminates over $\mathbb{K}$ and at the $f_i$'s as elements of
$\mathbb{K}[a,b,u,v,x,y]$. The polynomials $f_7,f_8,f_9$ have degree three in
$x,y$, while for $i=4,5,6$, the polynomials $\hat f_i$ have the form $\hat f_i=abvxy
f_i$, where $f_i$ is in $\mathbb{K}[a,b,u,v,x,y]$. The degree of $f_4,f_5,f_6$ in $x,y$ is
three.

Generally speaking, we are looking for specializations such that the corresponding
configuration gives rise to a proper realization of $G$.

\begin{lemma} \label{lm:u3v3}
If any of the equations $u=1$, $v=1$, $u=v$ holds then $\Delta$ is algebraic.
\end{lemma}
\begin{proof}
If $u=v=1$ then $\alpha=\beta$. Let $\Gamma$ be the cubic curve containing $\Delta'$
and $3_2$. The equation of $\Gamma$ can be computed explicitly, and one sees that if
$u=v=1$ then $3_3\in \Gamma$. As by Lemma \ref{lm:pencil} $\Gamma$ is invariant under
$\alpha=\beta$, we have $4_2,5_2,4_3,5_3 \in \Gamma$, too. $\Gamma$ cannot be completely
reducible since then, some $i_1$ would be collinear with some $j_2,k_2$. Suppose that
$\Gamma=\ell\cup C$ with line $\ell$ and irreducible conic $C$. Then $\ell,C$ are
$\alpha$-invariant and the $1_i$'s are in $3$. If $0_2\in 3$ then $0_3,1_3,2_3\in \ell$,
and all $2_j$'s are in $3$ and all $3_k$'s are in $\ell$. As $\{0_1,1_1,2_1\}$,
$\{0_2,1_2,2_2\}$, $\{0_3,1_3,2_3\}$, $\{3_2,4_2,5_2\}$, $\{3_3,4_3,5_3\}$ are all orbits
of $\langle \alpha \rangle$ and the lines $0_23_3, 1_24_3, 2_25_3$ are concurrent, we
have that $\{3_1,4_1,5_1\}$, $\{6_1,7_1,8_1\}$ are $\langle \alpha \rangle$-orbits
contained in $C$. Continuing the process, we conclude that $\Delta \subset \Gamma$ (which
is of course not possible). The same result is obtained if we start from $j_2\in C$ or
$k_3\in C$.

Suppose now that $\Gamma$ is irreducible. Denote by $\Gamma^*$ the set of nonsingular
points. The $\langle \alpha \rangle$-orbits are all cosets of a subgroup $H^*$ of
$(\Gamma^*,+)$ of order $3$. Then, simple arithmetic on $\Gamma^*$ yields that
$\{3_1,4_1,5_1\}$, $\{6_1,7_1,8_1\}$ are $H^*$ cosets of $\Gamma^*$. Repeating this
argument, we obtain $\Delta \subset \Gamma$ again.

It remains to show that any of the equations $u=1$, $v=1$, $u=v$ implies the
other two. For that we observe the following equations of rational expressions:
\begin{eqnarray*}
\left . \frac{f_6}{\det(3_2 \times 0_3,4_2\times 1_3,0_2\times 4_3)} \right|_{u=1} &=&
\frac{v-1}{ay-1}, \\
\left . \frac{f_6}{\det(0_2 \times 5_3,1_2\times 3_3,3_2\times 0_3)} \right|_{v=1} &=&
\frac{u-1}{(ua-x)by}, \\
\left . \frac{f_6}{\det(0_2 \times 3_3,1_2\times 4_3,3_2\times 0_3)} \right|_{u=v} &=&
\frac{v-1}{(b-y)ax}.
\end{eqnarray*}
In either case, the denominators at the left hand side cannot be zero as the corresponding
lines are not concurrent. This proves that one equation implies another one, and two
imply the third. This finishes the proof.
\end{proof}

The proof of the following lemma contains some elementary, but heavy computation. This
computation can be formally verified by any computer algebra dealing with Groebner bases
within a few seconds.

\begin{lemma} \label{lm:a3b3}
If $a^3=b^3=1$ then $v=1$. In particular, $\Delta$ is algebraic.
\end{lemma}
\begin{proof}
We observe that $a^3=b^3=1$ holds if and only if the lines $0_10_2$, $1_11_2$, $2_12_2$
and the lines $0_11_2$, $1_12_2$, $2_10_2$ are concurrent. (In other words, $\Delta'$
forms a dual $AG(2,3)$.) Remember that in this case, our coordinate system is chosen such
that $F(1,1,1)$ is the fixed point $0_10_2\cap 1_11_2 \cap 2_12_2$. As $0_1=(1,0,0)$ and
$0_3 \in O_1O_2=0_1F$, we have $b=1$. By conjugaction in $\mathbb{K}$, we can assume
$a=\omega$ w.l.o.g.

Now, we can find polynomials $s_i,t_i, p_1, p_2,q_1,q_2$, $i\in\{3,\ldots,8\}$, in the
indeterminates $a$, $b$, $x$, $y$, $u$, $v$ with integer coefficients  such that
\begin{eqnarray*}
\sum s_i f_i + p_1(a-\omega) +p_2(b-1) &=&18 xv(v-1)(vx-y^2)(vx-\omega y^2),\\
\sum t_i f_i + q_1(a-\omega) +q_2(b-1) &=&18 (v-1)(uy^3-v^2x^3) .
\end{eqnarray*}
Assume $v\neq 1$. Since $\det(0_1,3_2,4_2)=y^2-vx\neq 0$, we have $vx-\omega y^2=
uy^3-v^2x^3=0$. This implies $u=\omega^2xy$ and $v=\omega y^2/x$. Straightforward
computation shows that the collineation $\gamma$ given by the matrix
\[\left ( \begin{array}{ccc}\omega&0&0\\ 0&\omega^2&0\\0&0&1\end{array} \right )\]
fixes $0_1,1_1,2_1$ and maps the points $0_2,\ldots,5_2, 0_3,\ldots, 5_3$
to the points
\[1_2,2_2,0_2,\; 5_2,3_2,4_2,\; 1_3,2_3,0_3,\; 5_3,3_3,4_3,\]
respectively. As $5\cdot 2=6$ and $5\cdot 3=4$ in $G$, we have
\[\gamma(3_1)=\gamma(0_23_3\cap 1_24_3)=1_25_3\cap 2_23_3=4_1.\]
Similarly, $\gamma(4_1)=5_1$ and $\gamma(5_1)=3_1$. Thus, $\gamma$ permutes the lines
$3_10_3$, $4_11_3$, $5_12_3$ cyclically. As these lines intersect in $6_2$, $6_2$ is a
fixed point of $\gamma$, which is not possible.
\end{proof}


We are now prepared to prove the main result.

\begin{theorem}
$\Delta$ is algebraic.
\end{theorem}
\begin{proof}
We can consider the $f_i$'s as polynomials in the indeterminates $a$, $b$, $u$, $v$, $x$,
$y$. Fix the values $a,b,u,v$ and let $F_i(X,Y)$ be the polynomials in two variables such
that $F_i(x,y)=f_i(a,b,u,v,x,y)$. Define the linear series $L$ generated by the $F_i$'s.

Recall that $\beta$ is the collineation of order $2$ mapping the points $0_1$, $1_1$,
$2_1$, $3_2$, $4_2$, $5_2$, $3_3$, $4_3$, $5_3$ to $1_1$, $2_1$, $0_1$, $4_2$, $5_2$,
$3_2$, $5_3$, $3_3$, $4_3$, respectively. From the definition of the $f_i$'s one sees
that the substitution $X'=u/Y$,  $Y'=vX/Y$ induces a linear automorphism of $L$ of degree
$3$. We will denote this induced map by $\beta$, as well.

Define the polynomials
\begin{eqnarray*}
E_1=uvX+{\omega}^2uY^2+{\omega}vX^2Y, && E_2=vX^2+{\omega}^2uY+{\omega}Y^2X, \\
\bar E_1=uvX+{\omega}uY^2+{\omega}^2vX^2Y, && \bar E_2=vX^2+{\omega}uY+{\omega}^2Y^2X,
\end{eqnarray*}
and
\begin{eqnarray*}
Q_1= \omega F_7-F_8,  &&Q_2= \omega^2 F_4-F_5,\\
\bar Q_1= \omega^2 F_7-F_8, &&\bar Q_2=\omega F_4-F_5
\end{eqnarray*}
of $L$. Then $E_1, E_2, Q_1,Q_2$ are eigenvectors of $\beta$ with eigenvalue
$\omega uv/Y^3$ and $\bar E_1,\bar E_2,\bar Q_1,\bar Q_2$ are eigenvectors of $\beta$
with eigenvalue $\omega^2 uv/Y^3$. We have the following resultant values:
\[R_{E_1,E_2}(Y)=\omega^2 uvY(uv-Y^3)^2, R_{\bar E_1,\bar E_2}(Y)=\omega uvY(uv-Y^3)^2.\]
This shows that the intersection of $E_1=0, E_2=0$ and the intersection of $\bar E_1=0,
\bar E_1=0$ consist of the points $0_1,1_1,1_1$ (with multiplicity $0$) and the fixed
points of $\beta$ (with multiplicity $2$). In particular, $E_1,E_2$ and $\bar E_1,\bar
E_2$ are linearly independent.

Straightforward calculation shows that
\begin{eqnarray*}
Q_1=G_{11}E_1+G_{12}E_2, && Q_2=G_{21}E_1+G_{22}E_2,\\
\bar Q_1=\bar G_{11}\bar E_1+\bar G_{12}\bar E_2, && \bar Q_2=\bar G_{21}\bar E_1+\bar
G_{22} \bar E_2,
\end{eqnarray*}
where
\begin{eqnarray*}
G_{11}&=&(\omega^2 b+ab^2{\omega}+a^2)({\omega}^2+v{\omega}+u), \\
G_{12}&=&(b^2{\omega}+\omega^2 a+a^2b)(uv+{\omega}^2u+v{\omega}), \\
G_{21}&=&{\omega}(\omega^2 b+ab^2{\omega}+a^2)({\omega}+\omega^2 v+u), \\
G_{22}&=&(b^2{\omega}+\omega^2 a+a^2b)({\omega}u+\omega^2 v+uv), \\
\bar G_{11}&=&({\omega}b+\omega^2 ab^2+a^2)({\omega}+\omega^2 v+u), \\
\bar G_{12}&=&(a{\omega}+a^2b+\omega^2 b^2)({\omega}u+\omega^2 v+uv), \\
\bar G_{21}&=&{\omega}^2({\omega}b+\omega^2 ab^2+a^2)({\omega}^2+v{\omega}+u), \\
\bar G_{22}&=&(a{\omega}+a^2b+\omega^2 b^2)(uv+{\omega}^2u+v{\omega}).
\end{eqnarray*}

Assume that $Q_1,Q_2$ are linearly independent. Then $E_1,E_2 \in \langle Q_1,Q_2 \rangle
\leq L$. Therefore, $\Gamma_1\cap \cdots  \cap \Gamma_4$ is contained in the zero
set of $E_1=E_2=0$, a contradiction. We can similarly show that $\bar Q_1,\bar Q_2$ must
be linearly dependent. This implies
\begin{eqnarray*}
0&=&G_{11}G_{22}-G_{12}G_{21}\\
&=&(2+\omega^2)(b^2\omega+\omega^2a+ba^2)(\omega ab^2+\omega^2b+a^2)(u-v)(u-1)(v-1), \\
0&=&\bar G_{11}\bar G_{22}-\bar G_{12}\bar G_{21}\\
&=&(2+\omega)(a\omega+ba^2+\omega^2b^2)(\omega b+\omega^2ab^2+a^2)(u-v)(u-1)(v-1).
\end{eqnarray*}
The resultants of the polynomials $(b^2\omega+\omega^2a+ba^2)(\omega
ab^2+\omega^2b+a^2)$, $(a\omega+ba^2+\omega^2b^2)(\omega b+\omega^2ab^2+a^2)$ with
respect to $a,b$ are $9b^7(b^3-1)^6$ and $9a^7(a^3-1)^6$. Thus, $\Delta$ is algebraic
by Lemmas \ref{lm:u3v3} and \ref{lm:a3b3}.
\end{proof}

The Maple 13 program performing the computations of this section is attached in Appendix
\ref{app:c3xc3}. We use Buchberger's algorithm in order to explicitely construct the polynomials of
Lemma \ref{lm:a3b3}. Thus, any computer algebra which can do symbolical calculation with rational
polynomials can be used to verify the results. This convinces us about the correctness of our
computation.

\section{$G \cong \C_2 \times \C_4$} \label{sec_c2c4}

The main ingredient of the proof is Lame's Theorem \cite[Proposition 2.3]{knp2011}. A classical
{\em{Lame configuration}} consists of two triples of distinct lines in $PG(2,\mathbb{K})$,
say $\ell_1, \ell_2,\ell_3$ and $r_1,r_2,r_3$, such that no line from one triple passes
through the common point of two lines from the other triple. For $1\leq j,k\leq 3$, let
$R_{jk}$ denote the common point of the lines $\ell_j$ and $r_k$. There are nine such
common points, and they are called the points of the Lame configuration.
\begin{proposition}{\rm{Lame's Theorem.}}\,If eight points from a Lame configuration
lie on a plane cubic then the ninth also does.
\end{proposition}

The group $C_2\times C_4$ can be given by the multiplication table
\[\begin{array}{cccccccc}
1&2&3&4&5&6&7&8\\
2&1&4&3&6&5&8&7\\
3&4&1&2&7&8&5&6\\
4&3&2&1&8&7&6&5\\
5&6&7&8&2&1&4&3\\
6&5&8&7&1&2&3&4\\
7&8&5&6&4&3&2&1\\
8&7&6&5&3&4&1&2
  \end{array}.\]
The triple $\{i_1,j_2,k_3\}$ is collinear if and only if $i*j=k$.

The following $6$-tuples of collinear points form a Lam\'e configuration:
\begin{align*}
U_1=\{1_1,1_2,1_3\},\{3_1,5_2,7_3\},\{6_1,7_2,3_3\}+\{1_1,7_2,7_3\},\{3_1,1_2,3_3\},\{6_1,
5_2 , 1_3\},\\
U_2=\{1_1,3_2,3_3\},\{3_1,7_2,5_3\},\{6_1,5_2,1_3\}+\{1_1,5_2,5_3\},\{3_1,3_2,1_3\},\{6_1,
7_2, 3_3\},\\
U_3=\{1_1,1_2,1_3\},\{3_1,7_2,5_3\},\{8_1,5_2,3_3\}+\{1_1,5_2,5_3\},\{3_1,1_2,3_3\},\{8_1,
7_2, 1_3\},\\
U_4=\{1_1,3_2,3_3\},\{3_1,5_2,7_3\},\{8_1,7_2,1_3\}+\{1_1,7_2,7_3\},\{3_1,3_2,1_3\},\{8_1,
5_2, 3_3\}.
\end{align*}
Let $C$ be a cubic curve through the points
\begin{align*}
1_1,1_2,1_3,  3_1,7_2,7_3,  6_1,7_2,3_3.
\end{align*}
Then $|C\cap U_1|, |C\cap U_2| \geq 8$, hence, $C$ passes through the nineth points $3_2$
and $5_2$. It follows that $|C\cap U_3|, |C\cap U_4| \geq 8$. Thus, $C$ contains
\[U_1\cup U_2 \cup U_3 \cup U_4 =\{ 1_1,3_1,6_1,8_1,  1_2,3_2,5_2,7_2,
1_3,3_3,5_3,7_3\}.\]
It is straightforward to check that any of the following Lam\'e configurations intersects
$C$ in at least $8$ points:
\begin{align*}
\{1_1,5_2,5_3\},\{3_1,1_2,3_3\},\{5_1,3_2,7_3\}+\{1_1,3_2,3_3\},\{3_1,5_2,7_3\},\{5_1,1_2,
5_3\},\\
\{1_1,1_2,1_3\},\{3_1,5_2,7_3\},\{7_1,3_2,5_3\}+\{1_1,5_2,5_3\},\{3_1,3_2,1_3\},\{7_1,1_2,
7_3\},\\
\{1_1,5_2,5_3\},\{6_1,7_2,3_3\},\{8_1,2_2,7_3\}+\{1_1,7_2,7_3\},\{6_1,2_2,5_3\},\{8_1,5_2,
3_3\},\\
\{1_1,5_2,5_3\},\{6_1,4_2,7_3\},\{8_1,7_2,1_3\}+\{1_1,7_2,7_3\},\{6_1,5_2,1_3\},\{8_1,4_2,
5_3\},\\
\{1_1,1_2,1_3\},\{6_1,7_2,3_3\},\{8_1,3_2,6_3\}+\{1_1,3_2,3_3\},\{6_1,1_2,6_3\},\{8_1,7_2,
1_3\},\\
\{1_1,1_2,1_3\},\{6_1,3_2,8_3\},\{8_1,5_2,3_3\}+\{1_1,3_2,3_3\},\{6_1,5_2,1_3\},\{8_1,1_2,
8_3\}.
\end{align*}
Hence, $C$ contains the further points $5_1,7_1,2_2,4_2,6_3,8_3$. Finally, we consider
the Lam\'e configurations
\begin{align*}
\{1_1,1_2,1_3\},\{2_1,5_2,6_3\},\{6_1,2_2,5_3\}+\{1_1,5_2,5_3\},\{2_1,2_2,1_3\},\{6_1,1_2,
6_3\},\\
\{3_1,1_2,3_3\},\{4_1,7_2,6_3\},\{6_1,2_2,5_3\}+\{3_1,7_2,5_3\},\{4_1,2_2,3_3\},\{6_1,1_2,
6_3\},\\
\{1_1,5_2,5_3\},\{7_1,6_2,3_3\},\{8_1,3_2,6_3\}+\{1_1,6_2,6_3\},\{7_1,3_2,5_3\},\{8_1,5_2,
3_3\},\\
\{1_1,8_2,8_3\},\{5_1,3_2,7_3\},\{6_1,7_2,3_3\}+\{1_1,7_2,7_3\},\{5_1,8_2,3_3\},\{6_1,3_2,
8_3\},\\
\{1_1,1_2,1_3\},\{7_1,7_2,2_3\},\{8_1,2_2,7_3\}+\{1_1,2_2,2_3\},\{7_1,1_2,7_3\},\{8_1,7_2,
1_3\},\\
\{3_1,2_2,4_3\},\{5_1,1_2,5_3\},\{6_1,7_2,3_3\}+\{3_1,1_2,3_3\},\{5_1,7_2,4_3\},\{6_1,2_2,
5_3\}.
\end{align*}
As before, one sees that any of them has at least $8$ points in common with $C$, thus, $C$
passes through all the points of the embedding of $C_2\times C_4$.

\section{$G \cong\Alt_4$ ($p = 0$)} \label{sec_alt4}

Let the group $\Alt(4)$ be given on the underlying set $\{1,\ldots,12\}$
by the Cayley table
\[\begin{array}{ccccccccccccc}
&1&2&3&4&5&6&7&8&9&10&11&12\\
&2&1&4&3&7&8&5&6&12&11&10&9\\
&3&4&1&2&8&7&6&5&10&9&12&11\\
&4&3&2&1&6&5&8&7&11&12&9&10\\
&5&6&7&8&9&10&11&12&1&2&3&4\\
&6&5&8&7&11&12&9&10&4&3&2&1\\
&7&8&5&6&12&11&10&9&2&1&4&3\\
&8&7&6&5&10&9&12&11&3&4&1&2\\
&9&10&11&12&1&2&3&4&5&6&7&8\\
&10&9&12&11&3&4&1&2&8&7&6&5\\
&11&12&9&10&4&3&2&1&6&5&8&7\\
&12&11&10&9&2&1&4&3&7&8&5&6
\end{array}\]
We have that the points $1_1,\ldots,{12}_1$, $1_2,\ldots,{12}_2$, $1_3,\ldots,{12}_3$
of the complex projective plane form a realization of $\Alt(4)$, if for all
$i,j,k=1,\ldots,12$, $i_1,j_2,k_3$ are collinear if and only if $i*j=k$.

\begin{proposition} \label{th:noa4}
$\Alt(4)$ cannot be realized on the complex projective plane.
\end{proposition}
\begin{proof}
We see that $\{1,2,3,4\}$ is an elementary Abelian normal subgroup and $\{1,5,9\}$
is a subgroup. Without a  loss of generality, we can assume that
\[\begin{array}{lll}
1_1=[1, 0,0],& 1_2=[0, 1,0],& 1_3=[1,-1,0],\\
2_1=[0,1,1], &2_2=[1,0,1], &2_3=[0,0,1].  
\end{array}\]
As $3_1,1_2,3_3$ are collinear, we can take $3_1$ and $3_3$ in the form $3_1=[a,b,c]$ and
$3_3=[a,b+1,c]$. This enables us to compute the remaining points
\[\begin{array}{lll}
3_1=[a, b,  c], &3_2=[a-1,1+b,c], &3_3=[a, b+1,c], \\
4_1=[a-1,1+b,c-1], &4_2=[a, b,c-1], &4_3=[a-1,b,c-1].
\end{array}\]
These points indeed form a realization of $\C_2\times \C_2$. Similarly, we choose $5_1$ and $9_1$
as generic points $[d_1,d_2,1]$, $[d_4,d_5,1]$. Then $5_3, 9_3$ have the form $[d_1,d_3,1]$,
$[d_4,d_6,1]$, respectively and computation yields
\[5_2=[d_4+d_5-d_3,d_3,1], 9_2=[d_1+d_2-d_6,d_6,1]. \]
Define the points
\[\begin{array}{lll}
5_1 =  9_2 1_3 \cap 1_2 5_3 , & 5_2 =  9_1 1_3 \cap 1_1 5_3, & 5_3 =  5_1 1_2 \cap 1_1 5_2, \\
6_1 =  9_2 4_3 \cap 2_2 5_3 , & 6_2  =  9_1 2_3 \cap 4_1 5_3, & 6_3 =  5_1 2_2 \cap 4_1 5_2, \\
7_1 =  9_2 2_3 \cap 3_2 5_3 , & 7_2  =  9_1 3_3 \cap 2_1 5_3, & 7_3 =  5_1 3_2 \cap 2_1 5_2, \\
8_1 =  9_2 3_3 \cap 4_2 5_3 , & 8_2  =  9_1 4_3 \cap 3_1 5_3, & 8_3 =  5_1 4_2 \cap 3_1 5_2, \\
9_1 =  5_2 1_3 \cap 1_2 9_3 , & 9_2  =  5_1 1_3 \cap 1_1 9_3, & 9_3 =  9_1 1_2 \cap 1_1 9_2, \\
10_1 =  5_2 3_3 \cap 2_2 9_3 , & {10}_2  =  5_1 2_3 \cap 3_1 9_3, & {10}_3 =  9_1 2_2 \cap
3_1 9_2, \\
11_1 =  5_2 4_3 \cap 3_2 9_3 , & {11}_2  =  5_1 3_3 \cap 4_1 9_3, & {11}_3 =  9_1 3_2 \cap
4_1 9_2, \\
12_1 =  5_2 2_3 \cap 4_2 9_3 , & {12}_2  =  5_1 4_3 \cap 2_1 9_3, & {12}_3 =  9_1 4_2 \cap
2_1 9_2.
\end{array}\]
Let us denote by $d_{ijk}$ the determinant of the $3 \times 3$ matrix with rows
$i_1,j_2,k_3$. We define the sets $X=\{d_{ijk} \mid k=i*j\}$ and
\[\begin{array}{rcllllll}
Y&=&\{(1, 2, 6),& (9, 1, 10),& (1, 10, 8),& (5, 9, 2), & (9, 5, 2),& (1, 10, 2),\\
&&(1, 3, 11),& (1, 4, 12), & (1, 5, 7),& (9, 1, 12),& (9, 5, 4),& (5, 9, 4),\\
&&(1, 5, 6),& (1, 6, 11),& (1, 9, 12),& (1, 12, 4), & (1, 5, 8),& (1, 7, 3),\\ 
&&(1, 7, 12),& (5, 9, 3),& (9, 1, 11),& (1, 12, 7),& (1, 9, 11),& (5, 1, 6),\\ 
&&(5, 1, 7),& (1, 6, 2),& (1, 3, 7),& (1, 8, 4),& (1, 11, 6),& (5, 1, 8),\\
&&(9, 5, 3),& (1, 8, 10),&(1, 11, 3),& (1, 4, 8), &(1, 9, 10),& (1, 2, 10)\}.
\end{array}\]
For all $(i,j,k) \in Y$, $i\cdot j\neq k$, thus for any proper complex realization of
$\Alt(4)$, we can substitute complex numbers in the variables $a,b,c,d_1,\cdots,d_6$
such that all polynomials in $X$ are zero and all for all $(i,j,k) \in Y$, $d_{ijk}\neq
0$. Put $g=\Pi_{(i,j,k) \in Y} d_{ijk}$ and define 
\[ X'=\{ f/\gcd(f,g) \mid f\in X\}.\]
It is still true that $\Alt(4)$ has a realization if and only if one can substitue
complex numbers in $a,b,c,d_1,\cdots,d_6$ such that all polynomials in $X'$ vanish.
Groebner basis computation shows that the ideal generated by $X'$ contains the polynomials
$d_1-d_4$, $d_2-d_5$, implying $5_1=9_1$, a contradiction.
\end{proof}

The Groebner basis computation of this section is too heavy for most of computer programs of this
type. We found two programs which is able to compute the Groebner basis: the F4 algorithm \cite{F4}
in the computer algebra system \textsc{Maple 13} and the \textit{modStd} library \cite{modstd} of
\textsc{Singular} \cite{DGPS}. These programs do not store the cofactors of the Groebner bases,
hence one cannot verify the result symbolically. However, two completely different implementations
deliver the same result, thus, we can trust this computation as well. 

The \textsc{Maple 13} implementation is attached in Appendix \ref{app:alt4maple} and the
\textsc{Singular} implementation is attached in Appendix \ref{app:alt4sing}. The computations take
less that $3$ minutes, and less that $3$ hours, respectively.

\appendix

\section{Maple code for the case $G=C_3\times C_3$}
\label{app:c3xc3}

This appendix contains 

\begin{scriptsize}
\begin{verbatim}
########################################################
# Maple 13 program for computing with dual 3-nets
# G = C3 x C3
########################################################
# Preparation

with(LinearAlgebra):

isect:=proc(a,b,c,d) 
  evala(CrossProduct(CrossProduct(a,b),CrossProduct(c,d))): 
end proc:
idet:=proc(a,b,c,d,e,f) 
  evala(Determinant(<CrossProduct(a,b)|CrossProduct(c,d)|CrossProduct(e,f)>)): 
end proc:

alias(omega=RootOf(X^2+X+1));

########################################################
# Part 1: Constructing the points, the transformations and the equations.
# One defines the transformations, the base points and the equations which
# correspond to certain collinearities. The unknowns $a,b,u,v,x,y$ are seen as 
# fixed elements of the base field. 
# In the program, Px_i denotes the point $x_i$ of the dual $3$-net.

alpha:=<<0,1,0>|<0,0,1>|<1,0,0>>;
beta:=<<0,v,0>|<0,0,1>|<u,0,0>>;

P0_1:=<1,0,0>: P1_1:=<0,1,0>: P2_1:=<0,0,1>:
P3_2:=<x,y,1>; P4_2:=beta.P3_2; P5_2:=beta.P4_2;
P0_3:=<a,b,1>; P1_3:=(alpha^(-1)).P0_3; P2_3:=(alpha^(-1)).P1_3;   
                     # We turn Lambda_3 in the opposite direction!

P0_2:=isect(P0_1,P0_3,P1_1,P1_3);
P1_2:=isect(P0_1,P1_3,P1_1,P2_3);
P2_2:=isect(P0_1,P2_3,P1_1,P0_3);
P3_3:=isect(P0_1,P3_2,P1_1,P5_2);
P4_3:=isect(P0_1,P4_2,P1_1,P3_2);
P5_3:=isect(P0_1,P5_2,P1_1,P4_2);

f:=[  0,0,
  idet(P0_2,P3_3,P1_2,P4_3,P2_2,P5_3)/(a*b*v*x*y),  # P3_1
  idet(P0_2,P4_3,P1_2,P5_3,P2_2,P3_3)/(a*b*v*x*y),  # P4_1
  idet(P0_2,P5_3,P1_2,P3_3,P2_2,P4_3)/(a*b*v*x*y),  # P5_1
  idet(P3_2,P0_3,P4_2,P1_3,P5_2,P2_3),              # P6_1
  idet(P3_2,P1_3,P4_2,P2_3,P5_2,P0_3),              # P7_1
  idet(P3_2,P2_3,P4_2,P0_3,P5_2,P1_3)               # P8_1
]:
f:=factor(f):

########################################################
# Part 2 (Lemma 2.4): Any of u=1, v=1, u=v implies the other two equations.

factor(subs(u=1,f[6]/idet(P0_2,P4_3,P3_2,P0_3,P4_2,P1_3)));
factor(subs(v=1,f[6]/idet(P0_2,P5_3,P1_2,P3_3,P3_2,P0_3)));
factor(subs(u=v,f[6]/idet(P0_2,P3_3,P1_2,P4_3,P3_2,P0_3)));

########################################################
# Part 3 (Lemma 2.5): If a^3=b^3=1 then v=1.
# We can assume a=omega and b=1.

gb:=Groebner[Basis]([op(f),a-omega,b-1], plex(u,v,x,y,a,b),output=extended):
factor(gb[1][3]),factor(gb[1][5]);

# We construct the cofactors explicitely:
s:=evala(18*gb[2][3][1..8]): 
t:=evala(18*gb[2][5][1..8]): 
q:=evala(18*[gb[2][5][9],gb[2][5][10]]):
p:=evala(18*[gb[2][3][9],gb[2][3][10]]):

factor(add(f[i]*s[i],i=3..8)+p[1]*(a-omega)+p[2]*(b-1));
factor(add(f[i]*t[i],i=3..8)+q[1]*(a-omega)+q[2]*(b-1));

# This shows that all cofactors have coefficients in Z[omega]:
seq(denom(factor(_u)),_u in [op(s),op(t),op(p),op(q)]);

########################################################
# Part 4: Computation with the beta-invariant polynomials.
# From now on, we consider $a,b,u,v$ as fixed elements of the base field 
# and $X,Y$ as indetereminates. 
# We define the action of $\beta$ on the polynomial ring in two variables.

F:=map(_x->subs({x=X,y=Y},_x),f): map(_x->degree(_x,{X,Y}),F);

betaonpoly:=proc(U) return factor(subs({X=u*Y/(v*X),Y=u/X},U)): end proc:

# This shows that the nontrivial solutions of F[3]=F[4]=F[5]=0 
# and F[6]=F[7]=F[8]=0 are $\beta$-invariant:

map(_x->factor(_x), 
  [betaonpoly(F[3])/F[5],betaonpoly(F[4])/F[3],betaonpoly(F[5])/F[4]]
);
map(_x->factor(_x), 
  [betaonpoly(F[6])/F[7],betaonpoly(F[7])/F[8],betaonpoly(F[8])/F[6]]
);

# We define the E[i]'s, barE[i]'s, Q[i]'s and barQ[i]'s 
# and show that their curves are $\beta$-invariant:

E:=[u*v*X+omega^2*u*Y^2+omega*v*X^2*Y,v*X^2+omega^2*u*Y+omega*Y^2*X]:
barE:=[u*v*X+omega*u*Y^2+omega^2*v*X^2*Y,v*X^2+omega*u*Y+omega^2*Y^2*X]:

Q:= [omega * F[6]-F[7],omega^2 * F[3]-F[4]]:
barQ:= [omega^2 * F[6]-F[7],omega * F[3]-F[4]]:

seq(factor(betaonpoly(_u)/_u), _u in [op(E),op(Q)]);
seq(factor(betaonpoly(_u)/_u), _u in [op(barE),op(barQ)]);

# We define the G[i,j]'s and barG[i,j]'s. 
# We show that they are indeed coefficients of Q[i]'s and barQ[i]'s.

G[2,1]:=omega*(omega^2*b+a*b^2*omega+a^2)*(omega+omega^2*v+u);
G[2,2]:=(b^2*omega+omega^2*a+a^2*b)*(omega*u+omega^2*v+u*v);
G[1,1]:=(omega^2*b+a*b^2*omega+a^2)*(omega^2+v*omega+u);
G[1,2]:=(b^2*omega+omega^2*a+a^2*b)*(u*v+omega^2*u+v*omega);
barG[1,1]:=(omega*b+omega^2*a*b^2+a^2)*(omega+omega^2*v+u);
barG[1,2]:=(omega*a+a^2*b+omega^2*b^2)*(omega*u+omega^2*v+u*v);
barG[2,1]:=omega^2*(omega*b+omega^2*a*b^2+a^2)*(omega^2+v*omega+u);
barG[2,2]:=(omega*a+a^2*b+omega^2*b^2)*(u*v+omega^2*u+v*omega);

map(_x->evalb(factor(_x)), [
  Q[1]=G[1,1]*E[1]+G[1,2]*E[2],
  Q[2]=G[2,1]*E[1]+G[2,2]*E[2],
  barQ[1]=barG[1,1]*barE[1]+barG[1,2]*barE[2],
  barQ[2]=barG[2,1]*barE[1]+barG[2,2]*barE[2]
]);

# We compute the factors of the determinants of the G[i,j]'s and barG[i,j]'s:

map(_x->evalb(factor(_x)), [
  G[1,1]*G[2,2]-G[1,2]*G[2,1]=
    (2+omega^2)*(b^2*omega+omega^2*a+b*a^2)*(omega*a*b^2+omega^2*b+a^2)*(u-v)*(u-1)*(v-1),
  barG[1,1]*barG[2,2]-barG[1,2]*barG[2,1]=
    (2+omega)*(a*omega+b*a^2+omega^2*b^2)*(omega*b+omega^2*a*b^2+a^2)*(u-v)*(u-1)*(v-1)
]);

ra:=resultant(
  (-b+a*b^2*omega-omega*b+a^2)*(-b^2*omega+a+a*omega-b*a^2),
  (-omega*b+a*b^2+a*b^2*omega-a^2)*(a*omega+b*a^2-b^2-b^2*omega),
  a):
rb:=resultant(
  (-b+a*b^2*omega-omega*b+a^2)*(-b^2*omega+a+a*omega-b*a^2),
  (-omega*b+a*b^2+a*b^2*omega-a^2)*(a*omega+b*a^2-b^2-b^2*omega),
  b):
factor(ra/(b^3-1)^6);
factor(rb/(a^3-1)^6);
\end{verbatim}
\end{scriptsize}

\section{Maple code for the case $G=\Alt_4$}
\label{app:alt4maple}

This appendix contains the implementation of the computations of Section \ref{sec_alt4}, using the
F4 algorithm \cite{F4} in the computer algebra system \textsc{Maple 13}. This program does not store
the cofactors of the Groebner bases, hence one cannot verify the result symbolically. The
computation takes less than $3$ minutes.

\begin{scriptsize}
\begin{verbatim}
########################################################
# Maple 13 program for computing with dual 3-nets
# G = Alt(4)
########################################################
# Preparation

with(LinearAlgebra);
isect:=proc(a,b,c,d) 
  evala(CrossProduct(CrossProduct(a,b),CrossProduct(c,d))): 
end proc:

ct:=Matrix(
[ [ 1,2,3,4,5,6,7,8,9,10,11,12 ],
  [ 2,1,4,3,7,8,5,6,12,11,10,9 ],
  [ 3,4,1,2,8,7,6,5,10,9,12,11 ],
  [ 4,3,2,1,6,5,8,7,11,12,9,10 ],
  [ 5,6,7,8,9,10,11,12,1,2,3,4 ],
  [ 6,5,8,7,11,12,9,10,4,3,2,1 ],
  [ 7,8,5,6,12,11,10,9,2,1,4,3 ],
  [ 8,7,6,5,10,9,12,11,3,4,1,2 ],
  [ 9,10,11,12,1,2,3,4,5,6,7,8 ],
  [ 10,9,12,11,3,4,1,2,8,7,6,5 ],
  [ 11,12,9,10,4,3,2,1,6,5,8,7 ],
  [ 12,11,10,9,2,1,4,3,7,8,5,6 ] 
 ]);

d:=array(1..6);

########################################################
# Part 1: We define the points of the dual 3-net 
# using a,b,c,d[1],...d[6] as indeterminates.

P:=[  [<1,0,0>,<0,1,0>,<1,-1,0>],
      [<0,1,1>,<1,0,1>,<0,0,1>],
      [<a,b,c>,0,<a,1+b,c>],
      [0,0,0],
      [<d[1],d[2],1>,0,<d[1],d[3],1>],
      [0,0,0],
      [0,0,0],
      [0,0,0],
      [<d[4],d[5],1>,0,<d[4],d[6],1>],
      [0,0,0],
      [0,0,0],
      [0,0,0]  ];

# As P[4,1], P[1,2], P[4,3] are coll, we may assume wlog that 
# P[4,1]=<a,b,c> and P[4,3]=<a,1+b,c>. 
# Similar argument for P[5,3] and P[9,3], using the fact that 
# these points cannot have last coordinate 0.

P[3,2]:=evala(isect(P[3,1],P[1,3],P[1,1],P[3,3])/c);
P[4,2]:=evala(isect(P[3,1],P[2,3],P[2,1],P[3,3])/a);

P[4,1]:=evala(isect(P[3,2],P[2,3],P[2,2],P[3,3])/(1+b));
P[4,3]:=evala(isect(P[1,1],P[4,2],P[2,1],P[3,2])/(1+b-c));

P[5,2]:=isect(P[1,1],P[5,3],P[9,1],P[1,3]);
P[9,2]:=isect(P[1,1],P[9,3],P[5,1],P[1,3]);

########################################################

P[5,1]:=isect(P[9,2],P[1,3],P[1,2],P[5,3]):
P[5,2]:=isect(P[9,1],P[1,3],P[1,1],P[5,3]):
P[5,3]:=isect(P[5,1],P[1,2],P[1,1],P[5,2]):

P[6,1]:=isect(P[9,2],P[4,3],P[2,2],P[5,3]):
P[6,2]:=isect(P[9,1],P[2,3],P[4,1],P[5,3]):
P[6,3]:=isect(P[5,1],P[2,2],P[4,1],P[5,2]):

P[7,1]:=isect(P[9,2],P[2,3],P[3,2],P[5,3]):
P[7,2]:=isect(P[9,1],P[3,3],P[2,1],P[5,3]):
P[7,3]:=isect(P[5,1],P[3,2],P[2,1],P[5,2]):

P[8,1]:=isect(P[9,2],P[3,3],P[4,2],P[5,3]):
P[8,2]:=isect(P[9,1],P[4,3],P[3,1],P[5,3]):
P[8,3]:=isect(P[5,1],P[4,2],P[3,1],P[5,2]):

P[9,1]:=isect(P[5,2],P[1,3],P[1,2],P[9,3]):
P[9,2]:=isect(P[5,1],P[1,3],P[1,1],P[9,3]):
P[9,3]:=isect(P[9,1],P[1,2],P[1,1],P[9,2]):

P[10,1]:=isect(P[5,2],P[3,3],P[2,2],P[9,3]):
P[10,2]:=isect(P[5,1],P[2,3],P[3,1],P[9,3]):
P[10,3]:=isect(P[9,1],P[2,2],P[3,1],P[9,2]):

P[11,1]:=isect(P[5,2],P[4,3],P[3,2],P[9,3]):
P[11,2]:=isect(P[5,1],P[3,3],P[4,1],P[9,3]):
P[11,3]:=isect(P[9,1],P[3,2],P[4,1],P[9,2]):

P[12,1]:=isect(P[5,2],P[2,3],P[4,2],P[9,3]):
P[12,2]:=isect(P[5,1],P[4,3],P[2,1],P[9,3]):
P[12,3]:=isect(P[9,1],P[4,2],P[2,1],P[9,2]):


########################################################
# Part 2: We construct the polynomial identities.  

eqs:=[]:
for i from 1 to 12 do
  for j from 1 to 12 do
      aa:=Determinant(<P[i,1]|P[j,2]|P[ct[i,j],3]>):
      eqs:=[op(eqs),aa]:
  end do
end do:

########################################################
# Part 3: We filter out the nonzero factors of the equations.

nepos:=[
	[1, 2, 6], [9, 1, 10], [1, 10, 8], [5, 9, 2], 
	[9, 5, 2], [1, 10, 2], [1, 3, 11], [1, 4, 12], 
	[1, 5, 7], [9, 1, 12], [9, 5, 4], [5, 9, 4], 
	[1, 5, 6], [1, 6, 11], [1, 9, 12], [1, 12, 4], 
	[1, 5, 8], [1, 7, 3], [1, 7, 12], [5, 9, 3], 
	[9, 1, 11], [1, 12, 7], [1, 9, 11], [5, 1, 6], 
	[5, 1, 7], [1, 6, 2], [1, 3, 7], [1, 8, 4], 
	[1, 11, 6], [5, 1, 8], [9, 5, 3], [1, 8, 10], 
	[1, 11, 3], [1, 4, 8],[1, 9, 10], [1, 2, 10]
];
noneqs:=map(_x->Determinant(<P[_x[1],1]|P[_x[2],2]|P[_x[3],3]>),nepos):

noneqs:=mul(x,x in noneqs):

eqs:=select(x->x<>0,eqs): nops(eqs);
eqs_reduced:=map(x->factor(x/gcd(x,noneqs)),eqs):
map(degree,eqs)-map(degree,eqs_reduced); 

########################################################
# Part 4: We compute the Groebner basis of the corresponding ideal.
# This Groebner basis shows that d[1]=d[4], d[2]=d[5], d[3]=d[6]. 
# The computation takes less that 3 minutes using the F4 algorithm.

gb:=Groebner[Basis](eqs_reduced,tdeg(a,b,c,d[1],d[2],d[3],d[4],d[5],d[6]));
\end{verbatim}
\end{scriptsize}

\section{Singular code for the case $G=\Alt_4$}
\label{app:alt4sing}

This appendix contains the implementation of the computations of Section \ref{sec_alt4}, using the
\textit{modStd} library \cite{modstd} of \textsc{Singular} \cite{DGPS}. This program does not store
the cofactors of the Groebner bases, hence one cannot verify the result symbolically. The
computation takes less than $3$ hours. 

\begin{scriptsize}
\begin{verbatim}
////////////////////////////////////////////////////////
// Singular 3.1 program for computing with dual 3-nets
// G = Alt(4)
////////////////////////////////////////////////////////
// Preparation

LIB "modstd.lib";

intmat ct[12][12]=
   1,2,3,4,5,6,7,8,9,10,11,12 ,
   2,1,4,3,7,8,5,6,12,11,10,9 ,
   3,4,1,2,8,7,6,5,10,9,12,11 ,
   4,3,2,1,6,5,8,7,11,12,9,10 ,
   5,6,7,8,9,10,11,12,1,2,3,4 ,
   6,5,8,7,11,12,9,10,4,3,2,1 ,
   7,8,5,6,12,11,10,9,2,1,4,3 ,
   8,7,6,5,10,9,12,11,3,4,1,2 ,
   9,10,11,12,1,2,3,4,5,6,7,8 ,
   10,9,12,11,3,4,1,2,8,7,6,5 ,
   11,12,9,10,4,3,2,1,6,5,8,7 ,
   12,11,10,9,2,1,4,3,7,8,5,6;

ring r=0,(a,b,c,d(1..6)),dp;

////////////////////////////////////////////////////////
// Part 1: We define the points of the dual 3-net 
// using a,b,c,d[1],...d[6] as indeterminates.

// As P[4,1], P[1,2], P[4,3] are coll, we may assume wlog that 
// P[4,1]=<a,b,c> and P[4,3]=<a,1+b,c>. 
// Similar argument for P[5,3] and P[9,3], using the fact that 
// these points cannot have last coordinate 0.

list pt_data=
   1,0,0,         0,1,0,      1,-1,0,      // 1
   0,1,1,         1,0,1,      0,0,1,       // 2
   a,b,c,         0,0,0,      a,1+b,c,     // 3
   0,0,0,         0,0,0,      0,0,0,       // 4    
   d(1),d(2),1,   0,0,0,      d(1),d(3),1, // 5 
   0,0,0,         0,0,0,      0,0,0,       // 6
   0,0,0,         0,0,0,      0,0,0,       // 7
   0,0,0,         0,0,0,      0,0,0,       // 8
   d(4),d(5),1,   0,0,0,      d(4),d(6),1, // 9
   0,0,0,         0,0,0,      0,0,0,       // 10
   0,0,0,         0,0,0,      0,0,0,       // 11
   0,0,0,         0,0,0,      0,0,0;       // 12

////////////////////////////////////////////////////////
// Procedures for manipulating the points of the dual 3-net:

proc rpt(int x, int i)
{
   return(list(pt_data[(x-1)*9+(i-1)*3+1..(x-1)*9+(i-1)*3+3]));
}

proc setpoint(int x, int i, list u)
{
   pt_data[(x-1)*9+(i-1)*3+1]=u[1];
   pt_data[(x-1)*9+(i-1)*3+2]=u[2];
   pt_data[(x-1)*9+(i-1)*3+3]=u[3];
}

proc divpoint(int x, int i, poly p)
{
   pt_data[(x-1)*9+(i-1)*3+1]=pt_data[(x-1)*9+(i-1)*3+1]/p;
   pt_data[(x-1)*9+(i-1)*3+2]=pt_data[(x-1)*9+(i-1)*3+2]/p;
   pt_data[(x-1)*9+(i-1)*3+3]=pt_data[(x-1)*9+(i-1)*3+3]/p;
}

proc crossprod(list u,list v)
{
   return(list(u[2]*v[3]-u[3]*v[2],-u[1]*v[3]+u[3]*v[1],u[1]*v[2]-u[2]*v[1]));
}

proc isect(u,v,w,z) 
{
   return(crossprod(crossprod(u,v),crossprod(w,z))): 
}

proc detpoints(int x, int y, int z)
{
   matrix m[3][3] = 
      pt_data[(x-1)*9+0*3+1..(x-1)*9+0*3+3],
      pt_data[(y-1)*9+1*3+1..(y-1)*9+1*3+3],
      pt_data[(z-1)*9+2*3+1..(z-1)*9+2*3+3];
   return(det(m));
}

////////////////////////////////////////////////////////
// Part 2: We set the remaining points of the dual 3-net. 

setpoint(3,2,isect(rpt(3,1),rpt(1,3),rpt(1,1),rpt(3,3)));
divpoint(3,2,c); print(rpt(3,2));

setpoint(4,2,isect(rpt(3,1),rpt(2,3),rpt(2,1),rpt(3,3)));
divpoint(4,2,a); print(rpt(4,2));

setpoint(4,1,isect(rpt(3,2),rpt(2,3),rpt(2,2),rpt(3,3)));
divpoint(4,1,1+b); print(rpt(4,1));

setpoint(4,3,isect(rpt(1,1),rpt(4,2),rpt(2,1),rpt(3,2)));
divpoint(4,3,1+b-c); print(rpt(4,3));

////////////////////////////////////////////////////////
setpoint(5,2,isect(rpt(1,1),rpt(5,3),rpt(9,1),rpt(1,3)));
setpoint(9,2,isect(rpt(1,1),rpt(9,3),rpt(5,1),rpt(1,3)));

////////////////////////////////////////////////////////
setpoint(5,1,isect(rpt(9,2),rpt(1,3),rpt(1,2),rpt(5,3)));
setpoint(5,2,isect(rpt(9,1),rpt(1,3),rpt(1,1),rpt(5,3)));
setpoint(5,3,isect(rpt(5,1),rpt(1,2),rpt(1,1),rpt(5,2)));

setpoint(6,1,isect(rpt(9,2),rpt(4,3),rpt(2,2),rpt(5,3)));
setpoint(6,2,isect(rpt(9,1),rpt(2,3),rpt(4,1),rpt(5,3)));
setpoint(6,3,isect(rpt(5,1),rpt(2,2),rpt(4,1),rpt(5,2)));

setpoint(7,1,isect(rpt(9,2),rpt(2,3),rpt(3,2),rpt(5,3)));
setpoint(7,2,isect(rpt(9,1),rpt(3,3),rpt(2,1),rpt(5,3)));
setpoint(7,3,isect(rpt(5,1),rpt(3,2),rpt(2,1),rpt(5,2)));

setpoint(8,1,isect(rpt(9,2),rpt(3,3),rpt(4,2),rpt(5,3)));
setpoint(8,2,isect(rpt(9,1),rpt(4,3),rpt(3,1),rpt(5,3)));
setpoint(8,3,isect(rpt(5,1),rpt(4,2),rpt(3,1),rpt(5,2)));

setpoint(9,1,isect(rpt(5,2),rpt(1,3),rpt(1,2),rpt(9,3)));
setpoint(9,2,isect(rpt(5,1),rpt(1,3),rpt(1,1),rpt(9,3)));
setpoint(9,3,isect(rpt(9,1),rpt(1,2),rpt(1,1),rpt(9,2)));

setpoint(10,1,isect(rpt(5,2),rpt(3,3),rpt(2,2),rpt(9,3)));
setpoint(10,2,isect(rpt(5,1),rpt(2,3),rpt(3,1),rpt(9,3)));
setpoint(10,3,isect(rpt(9,1),rpt(2,2),rpt(3,1),rpt(9,2)));

setpoint(11,1,isect(rpt(5,2),rpt(4,3),rpt(3,2),rpt(9,3)));
setpoint(11,2,isect(rpt(5,1),rpt(3,3),rpt(4,1),rpt(9,3)));
setpoint(11,3,isect(rpt(9,1),rpt(3,2),rpt(4,1),rpt(9,2)));

setpoint(12,1,isect(rpt(5,2),rpt(2,3),rpt(4,2),rpt(9,3)));
setpoint(12,2,isect(rpt(5,1),rpt(4,3),rpt(2,1),rpt(9,3)));
setpoint(12,3,isect(rpt(9,1),rpt(4,2),rpt(2,1),rpt(9,2)));


////////////////////////////////////////////////////////
// Part 3: We define the nonzero polynomials.

intmat nz_pos[36][3]=
   1,2,6,    9,1,10,   1,10,8,    5,9,2,
   9,5,2,    1,10,2,   1,3,11,    1,4,12,
   1,5,7,    9,1,12,   9,5,4,     5,9,4,
   1,5,6,    1,6,11,   1,9,12,    1,12,4,
   1,5,8,    1,7,3,    1,7,12,    5,9,3,
   9,1,11,   1,12,7,   1,9,11,    5,1,6,
   5,1,7,    1,6,2,    1,3,7,     1,8,4,
   1,11,6,   5,1,8,    9,5,3,     1,8,10,
   1,11,3,   1,4,8,    1,9,10,    1,2,10;

list nz;
for (int i=1; i<=36; i++) 
{
  nz=insert(nz,detpoints(nz_pos[i,1],nz_pos[i,2],nz_pos[i,3]));
}

proc nonzero_reduction(poly p)
{
   for (int i=1; i<=36; i++)
   {
      p=p/gcd(p,nz[i]);
   }
   return(p);
}

////////////////////////////////////////////////////////
// Part 4: We construct the polynomial identities.  

poly p;
ideal I=0;
for (int i=1; i<=12; i++)
{
   for (int j=1; j<=12; j++)
   {
      p=detpoints(i,j,ct[i,j]);
      if (p!=0) { I=nonzero_reduction(p),I; }
   }
}

////////////////////////////////////////////////////////
// Part 5: We compute the Groebner basis of the corresponding ideal.
// This Groebner basis shows that d[1]=d[4], d[2]=d[5], d[3]=d[6]. 
// The computation takes less that 3 hours using the modStd method.

ideal J=modStd(I);
J;
\end{verbatim}
\end{scriptsize}

\vspace{0,5cm}\noindent {\em Authors' addresses}:

\vspace{0.2cm}\noindent G\'abor P\'eter NAGY\\ Bolyai Institute \\
University of Szeged \\ Aradi v\'ertan\'uk tere 1\\
H-6720 Szeged (Hungary).\\
E--mail: \texttt{nagyg@math.u-szeged.hu}

\vspace{0.2cm}\noindent Nicola PACE\\ Department of Mathematical Sciences  \\
Florida Atlantic University \\
777 Glades Road \\
Boca Raton, FL 33431, USA}. \\
E--mail: \texttt{nicolaonline@libero.it}
\end{document}